\documentclass[11pt]{amsart}        
\usepackage{amssymb,amscd,amsthm,amsxtra}
\usepackage{latexsym}
\input epsf
\begin{document}
\title{Spherical Averaged Endpoint Strichartz Estimates for The Two-dimensional Schr\"odinger Equations with Inverse Square Potential }
\author{I-Kun Chen}
\date{\today}

 \maketitle

 \oddsidemargin 0.0in

\newtheorem{thm}{Theorem}[section]
\newtheorem{defi}[thm]{Definition}
\newtheorem{lem}[thm]{Lemma}
\newtheorem{pro}[thm]{Proposition}
\newtheorem{rmk}[thm]{Remark}
\section{Introduction}
Strichartz estimates are crucial in handling local and global well-posedness problems for nonlinear dispersive equations (See\cite{boub}\cite{cazb}\cite{Taob}). For the Schrodinger equation below
\begin{equation}
\left\{
\begin{array}{l}i\partial_tu-\triangle u=0  \ \ \ \ u:\mathbb{R}^n\times \mathbb{R}^+\rightarrow \mathbb{C} \\ u(x,0)=u_0(x), \end{array}\right.
\end{equation}one considers estimates in mixed spacetime Lebesque norms of the type
\begin{equation}
\Vert u(x,t)\Vert_{L^q_tL^r_x}=(\int\Vert u(\cdot,t) \Vert_{L^r_x}^q dt)^{1/q}
.\end{equation}  Let us define the set of admissible exponents.
\begin{defi}
If $n $ is given, we say that the exponent pair $(q,r)$ is  admissible if $q,r\geq 2,\ (q,r,n)\neq(2,\infty,2) $ and they satisfy the relation
\begin{equation}
\frac2q+\frac n{r}=\frac n 2\ . \label{adm}
\end{equation}
\end{defi}
Under this assumption, the following estimates are known.
\begin{thm}
If  $(q,r,n)$ is admissible, we have the
estimates
\begin{equation}
\Vert u(x,t)\Vert_{L^q_tL^r_x}\leq \Vert u_0 \Vert_{L^2_x}
\ . \label{stest}\end{equation}
\end{thm}
From the scaling argument, or in other words dimensional analysis, we can see that the relation (\ref{adm}) is necessary for inequality (\ref{stest}) to hold.

There is a long line of investigation on this problem. The original work was done by Strichartz (see \cite{stri}\cite{stri1} \cite{stri2}). A more general result was done by Ginibre and Velo(See \cite{GV1}).
For dimension $n \geq 3$,  the endpoint cases $(q,r,n)=(2,\frac{2n}{n-2},n)$ was proved by Keel and Tao \cite{tk}.

The double endpoint $(q,r,n)=(2,\infty,2)$ is proved not to be true by Montgomery-Smith (see\cite{ms} ), even when we replace
$L_\infty$ norm with $BMO$ norm.
 However, it can be recovered in some special setting, for example Stefanov (See \cite{st}) and Tao (See \cite{Tao}). In particular, Tao replaces $L^\infty_x$ by a norm that takes  $L^2$ average over the angular variables then $L^\infty$ norm over the radial variable.

In the present work, I want to consider the end point estimates for  the   Schr\"odinger equation with inverse square potential,

\begin{equation} \label{eq1}
\left\{
\begin{array}{l}i\partial_tu-\triangle u+\frac{a^2}{|x|^2}u=0 \\ u(x,0)=u_0(x), \end{array}\right.
\end{equation} where $x\in\mathbb{R}^n$, and  initial data, $u_0\in L^2$. For $n\geq2$
the same Strichartz estimates as in Theorem 1 are proved by Planchon, Stalker, and Tahvidar-Zadeh(see \cite{pl1}). They did not cover the end point cases for $n=2$.

 We use
the same norm as Tao in \cite{Tao}.
We define the $L_\theta$ norm as follows.
\begin{defi}
\begin{equation}
\Vert f\Vert_{L_\theta}^2:=\frac1{2\pi}\int_0^{2\pi}|f(r\cos\theta, r\sin\theta)|^2 d\theta
.\end{equation}
\end{defi}
The main result in this paper is the following theorem.

\begin{thm} \label{mtheorem}
 For $x \in \mathbb{R}^2,\ a\geq 0$  , suppose $u(x,t)$ satisfies the following homogeneous initial value problem,
\begin{equation} \label{eq1}
\left\{
\begin{array}{l}i\partial_tu-\triangle u+\frac{a^2}{|x|^2}u=0 \\ u(x,0)=u_0(x), \end{array}\right.
\end{equation} then the following apriori  estimate holds
\begin{equation}
\Vert u \Vert_{L^2_t(L^\infty_rL_\theta)} \leq C\Vert u_0\Vert_{L^2(\mathbb{R}^2)}
.\end{equation}
\end{thm}

Let us consider the equation in polar coordinates. Write $v(r,\theta,t)=u(x,t)$
and $ f(r,\theta)=u_0(x)$. We have that $v(r,\theta,t)$ satisfies the equation below,
\begin{equation}
\left\{
\begin{array}{l}i\partial_tv-\partial_r^2 v -\frac1r\partial_rv-\frac1{r^2}\partial_\theta^2v+\frac {a^2 }{r^2}v=0\\ v(r,\theta,0)=f(r,\theta) \end{array}\right.
.\end{equation}
We write  the initial data as superposition of spherical harmonic
functions, as follows
\begin{equation*}
f(r,\theta)=\sum_{k\in \mathbb{Z}}f_k(r)e^{ik\theta}
.\end{equation*}
Using separation of variables, we can write $v$ as a superposition,
 \begin{equation*}
v(r,\theta,t)=\sum_{-\infty}^{\infty}e^{ik\theta}v_k(r,t)
,\end{equation*} where the  radial functions, $v_k$, satisfy the equations below
\begin{equation}\left\{
\begin{array}{l}i\partial_tv_k-\partial_r^2v_k-\frac1r\partial_rv_k+\frac{a^2+k^2}{r^2}v_k=0 ,\ \ \ \ k\in\mathbb{Z}\\
v_k(r,0)=f_k(r)\end{array}\right. \label{fkeq}.\end{equation}

\begin{rmk} Combining Tao's result in  \cite{Tao} with the equation (\ref{fkeq}) above, we can conclude that
Theorem \ref{mtheorem} is true for  special cases $a \in \mathbb{N}$ and $u$ is radially symmetric.  However, the analysis in \cite{Tao} does not apply to general cases. \end{rmk}
For fixed r, we take $L_\theta$ norm and from the orthogonality of spherical harmonics, we have
\begin{equation*}
\vert\vert
v(r,\theta,t)\vert\vert^2_{L_\theta}=\sum_{k\in\mathbb{Z}}|v_k(r,t)|^2
.\end{equation*}
We will prove the following lemma.
\begin{lem}
Suppose $v_k$ satisfies (\ref{fkeq}), for every $k\in\mathbb{Z}$ the following apriori estimate holds.
\begin{equation}
\int|v_k(r,t)|_{L^\infty_r}^2dt \leq C \int_0^\infty |f_k(r)|^2rdr \label{kineq}
,\end{equation}
where C is a constant independent of $k$ \label{lmkieq}
\end{lem}
The main theorem follows from the Lemma \ref{lmkieq} above because of the following observation,

\begin{eqnarray*}
||v(r,\theta,t)||^2_{L^2_tL^\infty_rL_\theta}=\int\left(\sup_{r>0}\{(\sum_{k\in\mathbb{Z}}|v_k(r,t)|^2)^{\frac12}\}\right)^2dt\\
\leq
\sum_{k\in\mathbb{Z}}\int\sup_{r>0}|v_k(r,t)|^2dt\\\leq\int_0^\infty
\sum_{k\in\mathbb{Z}}|f_k(r)|^2rdr=\int_0^{2\pi}\int_0^\infty |f(r)|^2rdrd\theta.
\end{eqnarray*}
The rest of the paper is devoted to the proof of Lemma \ref{lmkieq}.
\section{Hankel Tranform}
The main tools will be the Fourier and Hankel transforms.
We want to introduce certain well-known properties of Hankel transform which are necessary for the proof.
We consider the kth mode in spherical harmonic. Let $\nu(k)^2=a^2+k^2$.  We define the following
elliptic operator
\begin{equation}
A_\nu:=-\partial^2_r-\frac{1}{r}\partial_r+\frac{\nu^2}{r^2} .
\end{equation}
For fixed $k$,  we skip the $k$ in the notation for convenience.   Equation (\ref{fkeq}) becomes
\begin{equation}\left\{
\begin{array}{l}i\partial_tv+A_\nu v=0 ,\ \ \ \\
v(r,0)=f(r)\end{array}\right. \label{fkeq2}.\end{equation}
Next, we define the Hankel transform as follows.
\begin{equation}
\phi^\#(\xi):=\int^{\infty}_0 J_\nu (r \xi))\phi(r)rdr,
\end{equation}
where $J_\nu$ is the Bessel function of real order $\nu>-\frac12$ defined via,
\begin{equation}J_\nu(r)=\frac{({r/2})^\nu}{\Gamma(\nu+1/2)\pi^{1/2}}\int^1_{-1}e^{irt}(1-t^2)^{\nu-1/2}dt. \label{besseldef}\ .
\end{equation}

The following properties of the Hankel transform are well known, (See \cite{pl1})

\begin{pro}

\begin{eqnarray*}
&(i)& \ (\phi^\#)^\#=\phi\\
&(ii)& \ (A_\nu\phi)^\#(\xi)=|\xi|^2\phi^\#(\xi)\\
&(iii)&\int_0^\infty|\phi^\#(\xi)|^2\xi d\xi=\int^{\infty}_0|\phi(r)|^2 r dr.
\end{eqnarray*}
\label{pro1}\end{pro}
If we apply Hankel transform on  equation (\ref{fkeq2}), we obtain
\begin{equation}\left\{
\begin{array}{l}i\partial_t v^\#(\xi,t)+|\xi|^2v^\#(\xi,t)=0 ,\\ \ \ \ \
v^\#(\xi,0)=f^\#(\xi)\end{array}\right. \label{fkeq3}.\end{equation}
Solving the ODE and inverting the Hankel transform, we have the formula
\begin{equation}
v(r,t)=\int^{\infty}_0 J_\nu (s r)e^{is^2t}f^\#(s)sds
.\end{equation}
 The change of variables, $y=s^2$, implies
\begin{equation}\label{vexp}
v(r,t)=\frac12\int^{\infty}_0 J_\nu (r\sqrt{y})e^{iyt}f^\#(\sqrt{y})dy
.\end{equation}
Let us  define the function $h$ as follows
\begin{equation*}
h(y):=\left\{ \begin{array}{l}f^\#(\sqrt{y})\ \ y>0\\
0\ \ \ \ \ \ \ \ \ \  y\leq 0.\end{array} \right.
\end{equation*}
Then the expression in (\ref{vexp}) becomes
\begin{equation}
v(r,t)=\frac12\int_{\mathbb{R}} J_\nu (|r||y|^{1/2})h(y)e^{iyt}dy
.\label{fm1}\end{equation}
From the Proposition \ref{pro1}, we have
\begin{equation}
\int^\infty_{-\infty}|h(y)|^2dy=\frac12\int^\infty_0|f^\#(s)|^2sds=\frac12\int^\infty_0|f(\eta)|^2\eta
d\eta< \infty.
\end{equation} So, h is an $L^2$ function.
We will work with  $h(y)$  belonging to Schwartz class. These are $C^\infty$ functions that tend to zero faster than any polynomial at infinity, i.e.
\begin{equation*}
S(\mathbb{R})=\{f(x)\in C^\infty (x)|\sup_{x\in\mathbb{R}} \frac{d^\alpha f(x)}{d^\alpha x}<C^{\alpha\beta}x^{-\beta}\   \forall \
 \alpha,\beta\in\ \mathbb{N} \}.
\end{equation*}
 The general case of $h \in L_2$ follows by a density argument.
We use smooth cut off fountions to  partition the Bessel function $J_\nu$  as follwos,
 \begin{equation}
J_\nu(\eta)=m_\nu^0(\eta)+m_\nu^1(\eta)+\sum_{j\gg \log \nu}m_\nu^j(\eta),
\end{equation}
where $m_\nu^0$, $m_\nu^1$ and $m_\nu^j$ are supported on $\eta<\frac\nu{\sqrt{2}}$, $\eta \sim \nu$
 and $ \eta \sim 2^j  $ for $ j \gg \log_\nu$ respectively.
Let $J_\nu^k=\sum_0^k m_\nu^j$.
Equation (\ref{fm1}) holds in the sense that we can write
\begin{equation}
v(r,t)=\lim_{k\rightarrow \infty}\frac12\int_{\mathbb{R}} J_\nu^k (|r||y|^{1/2})h(y)e^{iyt}dy.
\end{equation}
 Substituting $h$ by the
inverse Fourier formula
\begin{equation*}
h(y)=\int_{\mathbb{R}}e^{i(\eta-t)y}\hat{h}(\eta-t)d\eta,
\end{equation*}
and changing the order of integration, we have

\begin{equation}
v(r,t)=\lim_{k\rightarrow\infty}\int_{\mathbb{R}}\left(\frac12\int_{\mathbb{R}} J^k_\nu
(\sqrt{|y|}|r|)e^{i\eta y}dy\right)\hat{h}_k(\eta-t)d\eta \label{fm2}
.\end{equation}
Let  us define the kernel below \begin{equation}\label{kndef} K_{\nu,r}^j(\eta)=\frac12\int_{\mathbb{R}} m_\nu^j
(\sqrt{|y|}|r|)e^{i\eta y}dy.
\end{equation}
For convenience, rename  $g(y)=\hat{h}(-y)$ and
define an operator
\begin{equation}
T^j_{\nu,r}[g](t)=(K^j_{\nu,r}*g)(t)
.\end{equation}
Since it is a convolution, it becomes a multiplication in Fourier space.
Thus, this operator has another equivalent expression
\begin{equation}
T^j_{\nu,r}[g](t)=\frac1{\sqrt{2\pi}}\int m_\nu^j(r|\xi|^\frac12)\widehat{g}(\xi)e^{i\xi t}d\xi \label{tjrdefeq}
.\end{equation}
Notice that both the kernel $K^j_{\nu,r}$ and the operator $T^j_{\nu,r}$ are functions of $\nu$.
We can rewrite equation (\ref{fm2}) in the following form,
\begin{equation}
v(r,t)=\lim_{k\rightarrow\infty}\sum_{j\leq k}T^j_{\nu,r}[g(\eta)](t)
.\end{equation}
The main theorem in this paper will follow from the lemma below.

\begin{lem}\label{3bounds}
For $g\in L^2$, $a\geq 0 $, $C_1$, $C_2$, $C_3$ independent of $\nu^2(k)=a^2+k^2$ $k\in \mathbb{N}$, the following estimates hold.
\begin{equation}
   \int_{\mathbb{R}}\sup_{r>0}|T_{\nu,r}^0(g)(t)|^2dt\leq C_1\int_\mathbb{R}|g(y)|^2dy
,\label{estt0}\end{equation}
\begin{equation}
   \int_{\mathbb{R}}\sup_{r>0}|T_{\nu,r}^1(g)(t)|^2dt\leq C_2\int_\mathbb{R}|g(y)|^2dy
,\label{estt1}\end{equation}
\begin{equation}
   \int_{\mathbb{R}}\sup_{r>0}|T_{\nu,r}^j(g)(t)|^2dt\leq C_32^{-\frac12j}\int_\mathbb{R}|g(y)|^2dy \ for\ j\gg\log\nu
\ \ \  \label{esttj}\end{equation}  .
\end{lem}

Notice $\Vert v \Vert^2_{L^2_t(L^\infty_rL_\theta)}$ can be bounded by the sum of the left hand side terms in Lemma\ref{3bounds} and the right hand side terms are summable. Thus, Lemma\ref{lmkieq}  follows.

We will refer to these three cases as low frequency, middle frequency, and high frequency respectively. We will prove inequalities (\ref{estt0}), (\ref{estt1}), and (\ref{esttj}) in the following sections.
\section{Estimates  for  Low Frequency}

 Our strategy is to estimate the kernel defined in (\ref{kndef}) and apply Hardy-Littlewood  maximal inequality in this case.
By changing variable $z:=r^2y$ in (\ref{kndef}), we can write
\begin{equation}
K^0_{\nu,r}(\eta)=\frac1{r^2}K^0_{\nu,1}(\frac{\eta}{r^2}),
\end{equation}
and therefore we have
\begin{equation}
\Vert K^0_{\nu,r}(\eta)\Vert_{L_1}=\Vert K^0_{\nu,1}(\eta)\Vert_{L_1}.
\end{equation}
We will prove the following estimate.
\begin{lem}
\label{lmk0} The kernel  $K^0_{\nu,1}(\eta)$ is bounded as follows,
\begin{equation}\label{k0est}
 |K^0_{\nu,1}(\eta)|\leq \Phi^0_\nu(\eta),
\end{equation}
where $\Phi^0_\nu$ is an even nonnegative  decaying $L^1$ function defined as follows.
\begin{equation}\label{phi0nu}\Phi^0_\nu= \left\{\begin{array}{lr} c(1+|\eta|)^{-(1+\nu/2)}& when \ 0 <\nu \leq 2 \\ C(\nu)(1+|\eta|)^{-2} &when\   2 < \nu ,\end{array}\right.
\end{equation}
where $C(\nu)$ is uniformly  bounded.
 \end{lem}
We can see $\Vert\Phi^0_\nu\Vert_{L^1}$ is finite for every $\nu$ from a direct calculation. Since $C(\nu)$ is uniformly bounded, $\Vert \Phi^0_\nu \Vert_{L^1}$ is uniformly bounded when $\nu >2$.
 For $0<\nu \leq 2$, $\Vert \Phi^0_\nu \Vert_{L^1}=4/\nu$.
However, since $\nu(k)^2=a^2+k^2$ are discrete, we can find a
universal $L_1$ bound for given $a\neq0$. Since $\Phi^0_\nu$ is an even nonnegative decaying function, we can use the property of
approximate identity and obtain.
\begin{equation}
\sup_{r>0} T^0_{\nu,r} [g](t) \leq \Vert \Phi^0_\nu \Vert_{L^1} M[g](t)
,\end{equation}
where $M[g](t)$ is  Hardy-Littlewood maximal function of $g$ at $t$,
defined as follows.
\begin{equation}
 M(g)(t)=\sup_{r>0}\frac1{|I(t,r)|}{\int_{I(t,r)}|g|dx}
,\end{equation} where $I(t,r)=(t-r,t+r)$.
Finally, we apply the Hardy-Littlewood  maximal inequality
\begin{equation}
\Vert M(F)\Vert_{L_p} \leq C(p) \Vert F \Vert_{L_p}\ \ \ \  1<p<\infty
\label{HLMInq}\end{equation}
to finish the proof.

 We will prove Lemma(\ref{lmk0}) case by case as presented in (\ref{phi0nu}).
\begin{proof} We need to prove $K^0_{\nu,1}(\eta)$ is bounded and decays with the the power advertised in (\ref{phi0nu}).
We first prove the decay of the tail.

Because $m^0_\nu$ is even, we have
\begin{equation} K_{\nu,1}^0(\eta)=\frac12\int_{\mathbb{R}} m_\nu^0
(\sqrt{|y|})e^{i\eta y}dy= \int_0^\infty m_\nu^0
(\sqrt{y})\cos (\eta y)dy\end{equation}
Integrate by parts to obtain
\begin{equation} K_{\nu,1}^0(\eta)= -\frac1{2\eta}\int_0^\infty{m^0_\nu}'
(y^{1/2})y^{-1/2}\sin (\eta y)dy\label{K0IBYP}\end{equation}
Differentiating  the expression of the Bessel function in (\ref{besseldef}), we can find the following recursive relation for Bessel functions .
\begin{equation}
J'_\nu(r)=\nu r^{-1}J_\nu(r)-J_{\nu+1}(r) \label{besid}
\end{equation}
From the definition of Bessel function (\ref{besseldef}), we can see
\begin{equation}
J_\nu(r)\sim \frac1{\Gamma(\nu+1)} (\frac{r}2)^\nu \ if\  r< \sqrt{\nu+1} . \label{asym}
\end{equation}
 Moreover, for all $r$ the following upper bound is true
\begin{equation}
J_\nu(r)\leq \frac c{\Gamma(\nu+1)} (\frac{r}2)^\nu. \label{jnuineq}
\end{equation}
Combining (\ref{besid}) and (\ref{asym}), the integrant in (\ref{K0IBYP})  behaves like $\sim\nu y^{\frac\nu2-1}$, when $y \ll 1+\nu$.

We will examine various cases of the parameter $\nu$.
\begin{itemize}
\item{Case 1: $0<\nu\leq 2$}
\end{itemize}
We break the integral into two parts, from $0$ to $|\eta|^{-\alpha}$ and the rest and integrate by parts the latter, i.e. we write
$ K_{\nu,1}^0(\eta)=I_1+B_2+I_2$, where
\begin{eqnarray*} I_1=-\frac1{2\eta}\int_0^{|\eta|^{-\alpha}} {m^0_\nu}'
(y^{1/2})y^{-1/2}\sin (\eta y)dy\\ B_2=-\frac1{2\eta^2}\cos(\eta |\eta|^{-\alpha}){m_\nu^0}'(|\eta|^{-\alpha/2})|\eta|^{\alpha/2}
\\I_2= -\frac1{4\eta^2}\int_{|\eta|^{-\alpha}}^\infty {m_\nu^0}^{\prime\prime}
(y^{1/2})y^{-1}\cos(\eta y)- {m^0_\nu}'
(y^{1/2})y^{-3/2}\cos (\eta y)dy,\end{eqnarray*} where $\alpha$  is a parameter to be determined later.

Estimate $I_1$ using the equation (\ref{asym}), we have $|I_1|\sim |\eta|^{-1-\alpha\frac\nu2}$.
 Taking the absolute value,  we have $|B_2| \sim \nu |\eta|^{-\alpha\frac\nu2+\alpha-2}$.
For $I_2$, we use the fact that Bessel function is the solution of the following differential equation
\begin{equation}
J_\nu^{\prime\prime}(r)+\frac1r J_\nu^\prime(r)+(1-\frac{\nu^2}{r^2})J_\nu(r)=0
.\end{equation}
Combining with the identity(\ref{besid}), we have
\begin{equation}
J_\nu^{\prime\prime}(r)=\frac1rJ_{\nu+1}(r)-(1+\frac\nu{r^2}-\frac{\nu^2}{r^2})J_\nu(r).\end{equation}
Using (\ref{jnuineq}), we can estimate the integrant in $I_2$ by  $c\nu(\nu-2)y^{\frac\nu2-2}$. Thus,  we have $|I_2| \leq c\nu |\eta|^{-\alpha\frac\nu2+\alpha-2}$.
To balance  the contribution from $I_1$, $B_2$, and $I_2$, we choose $\alpha=1$. Thus, we have  $K_{\nu,1}^0(\eta)< c\nu^{-(1+\frac\nu2)}$.
\begin{itemize}

\item{Case 2: $\nu<2$}
\end{itemize}
 We do not split the integral in this case.  We can integrate by parts twice without introducing boundary terms and obtain
\begin{eqnarray*}K^0_1(\eta)= -\frac1{4\eta^2}\int_0^\infty \left({m^0_\nu}^{\prime\prime}
(y^{1/2})y^{-1}\cos(\eta y)- {m^0_\nu}'
(y^{1/2})y^{-3/2}\cos (\eta y)\right)dy
.\end{eqnarray*}
Since $m_\nu^0$ is supported within $[0,\nu/\sqrt{2})$, the integral is bounded by ${\eta^{-2}}$ multiplied by a constant namely
 $C(\nu)=c(\nu-2){(\Gamma(\nu+1))^{-1}2^{\frac{-3\nu}2}}$, where $P$ is a polynomial with finite degree. Using the Stirling's formula
\begin{equation}
\Gamma(z)=\sqrt{\frac{2\pi}z}\left(\frac{z}e\right)^z\left(1+O(\frac1z)\right)\label{stirling}
,\end{equation}and observing that  $e<2^{3/2}$, we can see
 that $C(\nu)$ has a bound independent of $\nu$.

Now, we took care of the tail. The remaining task is to prove that $K_{\nu,1}^0(\eta)$ is bounded.  We take absolute value of the integrant  in (\ref{K0IBYP})
\begin{equation}
|K_{\nu,1}^0(\eta)| \leq \int |m^0_\nu(\sqrt{|y|})| dy.
\end{equation}
Since $m^0_\nu$ is a bounded function with a compact  support, we proved $K_{\nu,1}^1(\eta)$ is bounded for fixed $\nu$.  Furthermore if we apply (\ref{jnuineq}), we have
 \begin{equation}
|K_{\nu,1}^0(\eta)| \leq c\frac{\nu^3{\nu}^{\nu}}{\Gamma(\nu+1)2^{\frac{3\nu}2}}.
\end{equation}
Using the Stirling's formula (\ref{stirling}) again, we can show that there is a bound independent of $\nu$.
\end{proof}

\section{Estimates for Middle Frequency }
The goal is to prove the inequality (\ref{estt1}), namely
\begin{equation*}
\Vert T^1_{\nu,r}(g)(t)\Vert_{L^2_tL_r^\infty}\leq C \Vert g\Vert_{L^2}.
\end{equation*}
First, we want to estimate $L_r^\infty$  norm for fixed $t$.
Recall the equation (\ref{tjrdefeq}), we have
\begin{equation}
T^1_{\nu,r}(g)(t)=\frac1{\sqrt{2\pi}}\int m_\nu^1(r|\xi|^\frac12)\widehat{g}(\xi)e^{i\xi t}d\xi
\end{equation}
Since composing Fourier transform with inverse Fourier transform will form identity map , we have
\begin{equation}
T^1_{\nu,r_0}(g)(t)=\frac1{\sqrt{2\pi}^3}\int\int\int
m_\nu^1(r|\xi|^\frac12)\widehat{g}(\xi)e^{i\xi t}d\xi
\  e^{ir\rho} dr \ e^{-i\rho r_0} d\rho
.\label{inverT1r0}
\end{equation}

Using smooth dyadic decomposition, we write $\widehat{g}(\xi)=\sum \widehat{g_n}(\xi)$ where  $\widehat{g_n}$ is supported on $ (-2^{n+1},-2^{n-1})\bigcup(2^{n-1},2^{n+1})$. We will prove the following lemma.
\begin{lem}\label{lammagn}For $g_n \in L^2(\mathbb{R})$ such that $\widehat{g_n}$ supported on $ (-2^{n+1},-2^{n-1})\bigcup(2^{n-1},2^{n+1})$, we have the estimate
\begin{equation}
\Vert T^1_{r}(g_n)(t)\Vert_{L^2_tL_r^\infty}\leq C \Vert g_n\Vert_{L^2},\label{t1}
\end{equation} where $C$ is independent of $n$.
\end{lem}
\begin{proof}
  On the right hand side of  (\ref{inverT1r0}), we multiply and divide by $ \sqrt{b+\rho^2b^{-1}}$, where $b>0$ is a parameter to be chosen later.
We change the order of integration, and apply
Holder's inequality to obtain \begin{eqnarray}\label{t1exp}
 \ \ \ \ \ \ \ \ \ \ \ \ \ \ \ \ \ |T^1_{r_0}(g_n)(t)  |   \leq C\left( \int \frac {|e^{-i\rho r_0(t)}|^2}{(b+\rho^2b^{-1})}d \rho\right)^\frac12\cdot \\ \notag
\left(\int\left|\int\int
m_\nu^1(r  | \xi|^\frac12)\widehat{g_n}(\xi)e^{i\xi t} e^{ir\rho}d\xi dr\right|^2(b+\rho^2b^{-1})d\rho\right)^{\frac12}
\end{eqnarray}
Note that the first integral on the right hand side is $\pi$ for any $b>0$.  Thus, equation(\ref{t1exp}) reduces to
\begin{eqnarray}
 &\ &   \label{T1gn}  \ \ \  \Vert T^1_{r_0}(g_n)(t)\Vert_{L^\infty_{r_0}}\leq \\
 &C& \left(\int\left|\int\int
m_\nu^1(r|\xi|^\frac12)\widehat{g_n}(\xi)e^{i\xi t}d\xi  \notag e^{ir\rho}dr\right|^2(b+\rho^2b^{-1})d\rho\right)^{\frac12}
\end{eqnarray}
We name the integral on the right hand side of (\ref{T1gn}) as $l(t)$. We distribute the sum $(b +b^{-1}\rho^2)$ and write $l^2(t)=l_1(t)+l_2(t)$, where

\begin{eqnarray*}
l_1(t)=\int\left|\int\int b^{\frac12}
m_\nu^1(r|\xi|^\frac12)\widehat{g_n}(\xi)e^{i\xi t}d\xi e^{ir\rho}dr\right|^2d\rho
,\\
l_2(t)=\int\left|\int\int  b^{-\frac12}
\rho m_\nu^1(r|\xi|^\frac12)\widehat{g_n}(\xi)e^{i\xi t}d\xi e^{ir\rho}dr\right|^2d\rho
.\end{eqnarray*}
For $l_2(t)$, we
 integrate by parts with respect to $r$ to remove $\rho$ and obtain,
\begin{equation*}
l_2(t)=\int\left|\int\int  b^{-\frac12}
|\xi|^{\frac12}(m_\nu^1)'(r|\xi|^\frac12)\widehat{g_n}(\xi)e^{i\xi t}d\xi e^{ir\rho}dr\right|^2d\rho
.\end{equation*}
Using the Plancherel's theorem, we have
\begin{equation}
l_1(t)=\int\left|\int b^{\frac12}
m_\nu^1(r|\xi|^\frac12)\widehat{g_n}(\xi)e^{i\xi t}d\xi\right|^2dr
\end{equation}
\begin{equation}
l_2(t)=\int\left|\int b^{-\frac12}|\xi|^{\frac12}
(m_\nu^1)'(r|\xi|^\frac12)\widehat{g_n}(\xi)e^{i\xi t}d\xi \right|^2dr
.\end{equation}

  We square both sides of (\ref{T1gn}) and integrate overt $t$. Then, we change the order of integration with respect to $r,\ t$, and apply Plancherel's theorem again to obtain
\begin{eqnarray}
\Vert T^1_{r_0}(g_n)(t)\Vert_{L^2_tL_{r_0}^\infty}^2\leq C\int\int\left| b^{\frac12}
m_\nu^1(r|\xi|^\frac12)\widehat{g_n}(\xi)\right|^2 d\xi dr\\ +C\int\int\notag \left| b^{-\frac12}|\xi|^{\frac12}
m_\nu^1(r|\xi|^\frac12)\widehat{g_n}(\xi)\right|^2d\xi dr
.\end{eqnarray}
We change the variable $y=r|\xi|^{\frac12}$. We have
\begin{equation}
\Vert T^1_{r_0}(g_n)(t)\Vert_{L^2_tL_{r_0}^\infty}^2\leq C\int\left(\int \frac{b}{|\xi|^\frac{1}{2}}
|m_\nu^1(y)|^2+\frac{|\xi|^\frac{1}{2}}b |{m'_1}(y)|^2dy\right)|\widehat{g_n}(\xi)|^2d\xi
.\end{equation}
Use Lemma(\ref{jnunu})(see appendix) which implies
\begin{equation}
\begin{array}{cc}\int|
m_\nu^1(y)
 |^2 dy<C ,&|\int
{(m_\nu^1)}'(y)
   |^2dy<C.
\end{array}\end{equation}
Recall that the $\widehat{g_n}$ is supported on  $(-2^{n+1},-2^{n-1})\bigcup(2^{n-1},2^{n+1})$.
By choosing $b=2^{\frac n2}$, we complete the proof.
\end{proof}
 We proved (\ref{estt1}) for function has bounded support in Fourier domain described above. Now we are going to discuss the general case.
\begin{proof}{(\ref{estt1})} Suppose
$r_0(t)$ realizes at least half of the supremun at every $t$. Then,
it is enough to prove the inequality
\begin{equation} \label{r0in}
\int|T^1_{\nu,r_0(t)}(g)(t)|^2dt\leq C\Vert g \Vert_{L_2}^2.
\end{equation} We will prove (\ref{r0in}) for an arbitrary function $r_0(t)$. We dyadically decompose the range of
$r_0(t)$. The corresponding domains are defined as follows.
\begin{equation}
I_k=\{t|\ 2^k<r_0(t)\leq2^{k+1}\}\label{iqidex1}
\end{equation}
Since $m_\nu^1$ is supported on $(\nu/2,2\nu)$. We have
\begin{equation}
\frac{\nu}{2} < r_0|\xi|^{\frac12}< 2\nu.\label{iqidex2}
\end{equation} On $I_k$, by definition we have $2^{k}< r_0(t)\leq 2^{k+1}$. Combining (\ref{iqidex1}) and (\ref{iqidex2}), the integrant in the expression
(\ref{inverT1r0})
is nonzero only when
\begin{equation}
2\log_2\nu-2k-4<\log_2|\xi|<2\log_2\nu-2k+2 .\end{equation} As a
result, there are only 8 components in the dyadic decomposition in $\{\widehat{g}_n\}$ involved. When $t\in I_k $, we can rewrite (\ref{inverT1r0})
\begin{equation}
T^1_{\nu,r_0(t)}(g)(t)=\frac1{\sqrt{2\pi}}\int m_\nu^1(r_0|\xi|^\frac12)\sum_{n=n_0(k)}^{n_0+7}\widehat{g_n}(\xi)e^{i\xi t}d\xi
=\sum_{n=n_0(k)}^{n_0+7}T^1_{\nu,r_0(t)}(g_n)(t),\end{equation}
where $n_0(k)=\lfloor 2\log_2\nu-2k-4\rfloor$. Thus, use Cauchy-Schwartz inequality in finite sum to obtain
\begin{equation}
|T^1_{\nu,r_0(t)}(g)(t)|^2=\left|\sum_{n=n_0(k)}^{n_0+7}T^1_{\nu,r_0(t)}(g_n)(t)\right|^2\leq 8\sum_{n=n_0(k)}^{n_0+7}\left|T^1_{\nu,r_0(t)}(g_n)(t)\right|^2
.\end{equation}
Combine the above with the Lemma(\ref{lammagn}), we have
\begin{equation}
\int_{I_k}|T^1_{\nu,r_0(t)}(g)(t)|^2dt\leq C\sum_{n=n_0(k)}^{n_0+7}\Vert g_n \Vert_{L_2}^2.
\end{equation}
We sum  over $k$.
\begin{equation}
\int|T^1_{\nu,r_0(t)}(g)(t)|^2dt\leq C\sum_{k\in\mathbb{Z}}\sum_{n=n_0(k)}^{n_0+7}\Vert g_n \Vert_{L_2}^2.
\end{equation}
Note when we increase from $k$ to $k+1$, $n_0$ increases by $2$. As a result, every $n$ only appears four times. Thus
\begin{equation}
\int|T^1_{\nu,r_0(t)}(g)(t)|^2dt\leq 4C\sum_{n\in\mathbb{Z}}\Vert g_n \Vert_{L_2}^2\leq 4C\Vert g \Vert_2^2.
\end{equation}
This completes the proof.
\end{proof}
\section{Estimates for High Frequency}
The goal is to prove (\ref{esttj}), which is equivalent to
\begin{equation}
  \left\Vert\int_{\mathbb{R}}K_{r(t)}^j(t-\eta)g(\eta)d\eta\right\Vert_{L^2_t}\leq C2^{-\frac14j}\Vert g(y)\Vert_{L^2}
,\label{Kjest0}\end{equation}
 for an arbitrary function $r(t)$.
Using the $T^*T$ argument, we have the following lemma.

\begin{lem}
The following three inequalities are equivalent.
\begin{equation}
  \left\Vert\int_{\mathbb{R}}K_{\nu,r(t)}^j(t-\eta)g(\eta)d\eta\right\Vert_{L^2_t}\leq C2^{-\frac14j}\Vert g(y)\Vert_{L^2}
,\ \ \forall g \in L^2(\mathbb{R}^1)\label{Kjest}\end{equation}
\begin{eqnarray}
\left\Vert \int_{\mathbb{R}}K_{\nu,r(t)}^j(t-\eta)F(t) dt\right \Vert_{L^2}  \leq C
2^{-\frac14j}\Vert F \Vert _{L^2} \ \ \ \forall F \in L^2(\mathbb{R}^1) \label{ttmid}  \end{eqnarray}
\begin{eqnarray}\label{ttest2}
\left\Vert
\int_{\mathbb{R}}\int_{\mathbb{R}}K^j_{\nu, r(t)}(t-\eta)\overline{K^j_{\nu,r(t')}(t'-\eta)}d\eta
F(t')dt' \right\Vert_{L^2}  \\ \notag \leq C2^{-\frac12j}||F||_{L^2} ,\ \forall F \in L^2(\mathbb{R}^1)\end{eqnarray}
\end{lem}
\begin{proof}
Suppose we have (\ref{Kjest}), we want to show it implies (\ref{ttmid}).
We multiply the integrant on the left hand side of (\ref{Kjest})
with arbitrary $L^2$ function F(t), then  integrate over $t$, $\eta$. We
 apply Holder's inequality and (\ref{Kjest}) to obtain
\begin{equation}
\int \int_{\mathbb{R}}K_{\nu,r(t)}^j(t-\eta)F(t) dt g(\eta)d\eta \leq Ce^{-\frac14j}\Vert g \Vert_{L^2} \Vert F\Vert_{L^2}.
\end{equation}
Use the property that $L^2$ is self-dual, i.e. \begin{equation}
\Vert h \Vert_{L^2}=\sup_{f \in L^2}\frac{\int f(t)h(t)dt}{\Vert
f\Vert_{L^2}}.
\end{equation} We obtain (\ref{ttmid}).
Using
the same argument again, we can prove
$(\ref{Kjest})\Longleftrightarrow (\ref{ttmid})$.

Suppose we have  (\ref{ttmid}), we will show that (\ref{ttest2}) holds .
We multiply the integrant on the left hand side of (\ref{ttest2}) with
arbitrary $L^2$ function $G(t)$ and integrate over $\eta$, $t'$, and $t$. We change the order of integration and apply Holder's inequality and (\ref{ttmid}) to obtain
\begin{eqnarray*}
|\int\int\int K^j_{\nu,r(t)}(t-\eta)\overline{K^j_{\nu,r(t')}(t'-\eta)}d\eta
F(t')dt'G(t)dt|\\ \leq  \Vert\int K^j_{\nu,r(t)}(t-\eta)G(t)dt \Vert_{L^2} \Vert \int \overline{K^j_{\nu,r(t')}(t'-\eta)} F(t')dt' \Vert_{L^2}
\\
\leq Ce^{-\frac12j} \Vert G\Vert_{L^2}\Vert F\Vert_{L^2}
, \end{eqnarray*} which implies (\ref{ttest2}) by duality.

Suppose (\ref{ttest2}) holds.  We multiply the integrant on the left hand
side of (\ref{ttest2}) with complex conjugate of $F(t)$, $\overline{F(t)}$ , integrate over  $\eta$, $t'$, and $t$, apply Holder's inequality
and (\ref{ttest2}), we obtain (\ref{ttmid}). This completes the proof.
\end{proof}
Thus, to prove (\ref{Kjest0}), I have to to prove (\ref{ttest2}).
Inequality (\ref{ttest2}) will follow from the following estimate.
\begin{lem}\label{lemtt}
For any $a,b>0$, $t$, $t'\in \mathbb{R}$ we have
\begin{equation}
|\int K^j_a(t-\eta)\overline{K^j_b(t'-\eta}d\eta|<\frac1{a^2}\Phi_j(\frac{|t-t'|}{a^2})\label{TTest}
\end{equation}
where $\Phi_j$ is even non-increasing non-negative function with
\begin{equation}
 \Vert\Phi_j\Vert_{L_1}=\Vert \frac1{a^2}\Phi_j(\frac{y}{a^2})\Vert_{L_1(y)} \leq C2^{-\frac1{2}j}
\end{equation}
\end{lem}
 The estimate (\ref{TTest}) does not depend on $b$. And, $\Phi_j$ is even non-increasing non-negative. Thus, we have
\begin{eqnarray*}
|\int_{\mathbb{R}}\int_{\mathbb{R}}K^j_{r(t)}(t-\eta)\overline{K^j_{r(t')}(t'-\eta)}d\eta
F(t')dt' |\leq \int_{\mathbb{R}}\frac1{r(t)^2}\Phi_j(\frac{|t-t'|}{r(t)^2})
|F(t')|dt' \\ \leq  \sup_{r>0}  \int_{\mathbb{R}}\frac1{r^2}\Phi_j(\frac{|t-t'|}{r^2})|F(t')|dt'
 \leq \Vert \Phi_j\Vert_{L^1} M(F)(t) \leq C 2^{-\frac{j}2}M(F)(t)
,\end{eqnarray*} where $M(F)$ is Hardy-Littlewood maximal function of $F$. We apply Hardy-Littlewood maximal inequality (\ref{HLMInq})
to finish the proof of (\ref{esttj}).
\begin{proof}(Lemma \ref{lemtt})
The kernel is in the form of inverse Fourier transform
\begin{equation}
K^j_{\nu,r}(\eta)=\int_{\mathbb{R}}m_\nu^j(r|y|^{1/2})e^{i\eta y}dy
\label{kjakjbex}.\end{equation}
By Plancherel's theorem, we have
\begin{equation}
\int K^j_{\nu,a}(t-\eta)\overline{K^j_{\nu,b}(t'-\eta)}d\eta=\int m_\nu^j(a|y|^{1/2})\overline{m_\nu^j(b|y|^{1/2})}e^{i(t-t')y}dy
.\end{equation}
From the standard asymptotic of Bessel functions (see \cite{stein}), we have
\begin{equation}
m_\nu^j(\xi)=\sum_{\pm}2^{-j/2}e^{\pm i \xi} \psi^{\pm}_j(2^{-j}\xi)
.\label{mjexpres} \end{equation} where $\psi^{\pm}_j(\xi)$ are
function supported on $|\xi|\sim1$ and bounded  uniformly in $j$, $\nu$ .
 We can rewrite the right hand side of (\ref{kjakjbex}) as a finite number of expressions of the form
\begin{equation}
2^{-j}|\int e^{i(\pm a \pm b)|y|^{1/2}}e^{i(t-t')y}\psi^\pm_j(2^{-j}a|y|^{1/2})\psi^\pm_j(2^{-j}b|y|^{1/2}) dy|
\label{kjaymp},\end{equation}
where $\pm$ sign need not agree.
Since the bump functions are supported on $(\frac12,2)$, this expression is identically  zero except when $\frac14<\frac ba<4$. Let $\alpha=\frac ba$
The expression in (\ref{kjaymp}) becomes,
\begin{equation}
2^{-j}|\int e^{i(\pm 1 \pm \alpha)a|y|^{1/2}}e^{i(t-t')y}\psi^\pm_j(2^{-j}a|y|^{1/2})\psi^\pm_j(2^{-j}\alpha a|y|^{1/2}) dy|
\end{equation}Let $s=(t-t')$.
We name the sum of this expression $\phi^j_{a,\alpha}(s)$. By changing variable $y=a^2y$, we can see that
\begin{equation}
\phi^j_{a,\alpha}(s)=\frac1{a^2}\phi^j_{1,\alpha}\left(\frac{s}{a^2}\right).
\end{equation}
By letting $s'=\frac s {a^2}$, we have
\begin{equation}
\Vert \phi^j_{a,\alpha}(s) \Vert_{L_1}=\Vert \phi^j_{1,\alpha}(s) \Vert_{L_1}.
\end{equation}Now,
$\phi^j_{1,\alpha}(s)$ is finite sum of the following expressions
\begin{equation}
2^{-j}\left|\int e^{i(\pm 1 \pm
\alpha)|y|^{1/2}}e^{i(s)y}\psi^\pm_j(2^{-j}|y|^{1/2})\psi^\pm_j(2^{-j}\alpha
|y|^{1/2}) dy\right| \label{Kj3}.\end{equation} Changing  the
variable $z=2^{-j}|y|^{1/2}$, the expression in (\ref{Kj3}) becomes
\begin{equation}
2^{j+2}\left|\int_0^\infty e^{i(\pm 1 \pm
\alpha)2^jz+s2^{2j}z^2}\psi^\pm_j(z)\psi^\pm_j(\alpha z) zdz\right|
\end{equation}
We will prove the following lemma.
\begin{lem} \label{lmphi}
$\phi^j_{1,\alpha}$ is controlled by the following function
\begin{equation}
\Phi_j(s)=C\left\{ \begin{array}{lr} 2^{j} &  when\ 0< s\leq 2^{-2j} \\s^{-1/2} & when\ 2^{-2j}<s<40\ 2^{-j}\\2^j(2^{2j}s)^{-10}& otherwise . \end{array} \right.
\end{equation}
\end{lem}
Notice that we can estimate directly and get $\Vert \Phi_j \Vert_{L_1}\leq C2^{-\frac{j}2}$.\end{proof}

The remaining task is to prove the lemma (\ref{lmphi}).
\begin{proof}(Lemma \ref{lmphi})
Take absolute value of the integrant to obtain the bound $2^j$. We will use stationary phase technique to prove the other two estimates. We call the function on the index of exponential phase.
If we differentiate the phase, we get
$
2\ 2^{2j}sz+(\pm1\pm\alpha)2^j
$.
Note $\alpha$ is between $1/4$ and $4$, and the product of bump functions  is supported on $(1/8,8)$. Suppose $s>20\ 2^{-j}$,
the derivative is never zero on the support. We get the bound $2^j(2^{2j}s)^{-10}$ from non-stationary phase analysis. Otherwise,
we note the second derivative of the phase, namely $
2^{2j+1}s$, is not zero. By stationary phase analysis we obtain the bound $s^{-1/2}$.

\end{proof}

\section{Appendix: Estimate of Bessel Function around $\nu$ }
\begin{lem}\label{jnunu}
\par For the Bessel function $J_\nu$ of positive order $\nu$ and
when $\frac12\nu \leq r \leq 2\nu$, we have the following estimates
\begin{eqnarray}
J_\nu(r)&\leq& C\nu^{-\frac13}(1+\nu^{-\frac13}|r-\nu|)^{-\frac14}\label{Corput1}\\
J_\nu^{\prime}(r)& \leq & C\nu^{-\frac12}\label{Corput2}.
\end{eqnarray}
\end{lem}
\begin{proof}
We have an integral representation for the Bessel function of order $\nu>-\frac12$ (see\cite{wa}).
\begin{equation}
J_\nu(r)=A_\nu(r)-B_\nu(r),
\end{equation}
where \begin{equation}
A_\nu(r)=\frac1{2\pi}\int^\pi_{-\pi}e^{-i(r\sin\theta-\nu\theta)}d\theta
\label{Anu}\end{equation}
\begin{equation}
B_\nu(r)=\frac{\sin(\nu\pi)}\pi\int^\infty_0e^{-\nu t -r\sinh(t)}dt
.\end{equation}
 We can see \begin{equation}B_\nu(r)<\frac{\sin(\nu\pi)}\pi\int^\infty_0e^{-\nu t }dt<C\nu^{-1}.\end{equation} So, we only need to estimate $A_\nu$. We will accomplish this using stationary phase for  two different cases $r>\nu$  and $r\leq \nu$.
 Let us consider the case when $r>\nu$ first.
Call the phase in (\ref{Anu}) $\phi(\theta)=r\sin\theta-\nu \theta$.
 We differentiate the phase,  $\phi'(\theta)=(r\cos\theta-\nu)$. We find $\phi'=0$ at $\pm\theta_0$, where  $\theta_0=\cos^{-1}(\frac \nu r)$.
In order to obtain the estimate, we will break the integral into a small neighborhood around these points and the rest, that is
\begin{eqnarray}
N_\varepsilon&:=&\{\theta: \ | \theta\pm\theta_0|<\varepsilon\}\\
S_\varepsilon&:=&[-\pi,\pi]/ N_\varepsilon
\end{eqnarray}
Since the integrant is in (\ref{Anu}) is bounded, we have
\begin{equation}
|\int_{N_\varepsilon}e^{-i(r\sin\theta-\nu\theta)}d\theta|<c\varepsilon
.\end{equation}
On $S_\varepsilon$, we integrate by parts,
\begin{eqnarray*}
\int_{S_\varepsilon}e^{-i(r\sin\theta-\nu\theta)}d\theta=\left. \frac{e^{i(r\sin\theta-\nu\theta)}}{i(r\cos\theta-\nu)}\right|^{\{{\pi,\theta_0\pm\varepsilon}\}}_{\{{-\theta_0\pm\varepsilon},-\pi\}}
+\int_{S_\varepsilon}\frac{e^{i(r\sin\theta-\nu\theta)}r\sin\theta}{i(r\cos\theta-\nu)^2}d\theta.\end{eqnarray*}
All terms in the expression above are controlled by ${c}{|r\cos(\theta_0\pm\varepsilon)-\nu|^{-1}}$.
We want to balance the contribution from $N_\varepsilon$ and $S_\varepsilon $ by choosing proper $\varepsilon$, such that
\begin{equation}
\varepsilon \sim |r\cos(\theta_0\pm\varepsilon)-\nu|^{-1}
.\end{equation}
 Using trigonometric identities, $\cos(\theta_0 \pm \varepsilon)=\cos (\theta_0)\cos( \varepsilon) \mp \sin(\theta_0)\sin(\varepsilon) $,
and the definition of $\theta_0$, we have,
\begin{equation}
|r\cos(\theta_0\pm\varepsilon)-\nu|=|\nu\cos\varepsilon-\sqrt{r^2-\nu^2}\sin\varepsilon-\nu| \label{phasenu}.
\end{equation}
When $\varepsilon$ is small, (\ref{phasenu}) is approximately $\frac \nu2 \varepsilon^2+\varepsilon\sqrt{r^2-\nu^2}$. Thus, we have the two estimates
\begin{eqnarray}
|r\cos(\theta_0\pm\varepsilon)-\nu|^{-1} \leq 2\nu^{-1}\varepsilon^{-2} \label{ineq1}
,\\|r\cos(\theta_0\pm\varepsilon)-\nu|^{-1} \leq \varepsilon^{-1}(r^2-\nu^2)^{-\frac12}\label{ineq2}
.\end{eqnarray}
When $r-\nu$ is small, (\ref{ineq1}) is sharper.  We pick $\varepsilon \sim \nu^{-\frac13}$. When $r-\nu$
is big, (\ref{ineq2}) is sharper. We pick optimal $\varepsilon \sim (r^2-\nu^2)^{-\frac14}$. Since $ r\leq2\nu$
, we have
 $|(r^2-\nu^2)^{-\frac14}|<(3\nu)^{-\frac14}(r-\nu)^{-\frac14}$. Thus, we have proven (\ref{Corput1}) for the case $r\geq \nu$.

Now, we will discuss the case when $r\leq \nu$.
When $ \nu-\nu^{-\frac13}<r<\nu$, we follow the analysis above by choosing $\theta_0=0$, $\varepsilon=\nu^{-\frac13}$.
We have an estimate $J_\nu(r)\leq C \nu^{-\frac13}$. When $\nu-\nu^{-\frac13}>r$, we use non-stationary phase,
we obtain $J_\nu(r) \leq C\frac1{|\nu-r|} \leq C\nu^{-\frac14}|\nu-r|^{-\frac14}$.

The remaining task is to prove (\ref{Corput2}).
Considering the derivative of $J_\nu$, we  can show \begin{equation}|B^{\prime}_\nu(r)|=|\frac{\sin(\nu\pi)}\pi\int^\infty_0e^{-\nu t -r\sinh(t)}\sinh(t)dt|\leq c|\int^\infty_0e^{-\nu t }dt| \leq \frac{c}{\nu}.\end{equation} So, we only need to estimate
\begin{equation}
A^\prime_\nu(r)=\frac1{2\pi}\int^\pi_{-\pi}e^{i(r\sin\theta-\nu\theta)}i\sin\theta d\theta
\end{equation}
When $r>\nu$, break the integral into two as we did in the previous case.
\begin{equation}
|\frac1{2\pi}\int_{N_\varepsilon}e^{i(r\sin\theta-\nu\theta)}i\sin\theta d\theta|<C\varepsilon
\end{equation}
Integrate by parts for the integral on $S_\varepsilon$,
and use trigonometric identity and Taylor expansion. Then,  we can find it is controlled by
\begin{equation}
\frac{c_1\sqrt{\frac{r-\nu}{\nu}}+c_2\varepsilon}{\nu\frac{\varepsilon^2}{2}+c_3\sqrt{\nu}\sqrt{r-\nu}}
\end{equation}
If we balance between integral on $N_\varepsilon$ and $S_\varepsilon$, we get optimal $\nu^{-\frac12}$ .
For the case $r\leq\nu$, we can apply similar ideas  as in the proof of (\ref{Corput1}).
\end{proof}

\end{document}